\documentclass[11pt,leqno]{article}
\usepackage[margin=1in]{geometry} 
\usepackage{amssymb,amsfonts,amsmath,bbm,mathrsfs,stmaryrd,mathtools}
\usepackage{xcolor}
\usepackage{url}

\usepackage{extarrows}

\usepackage[shortlabels]{enumitem}
\usepackage{tensor}

\usepackage{xr}
\usepackage[T1]{fontenc}

\usepackage[colorlinks,
linkcolor=black!75!red,
citecolor=blue,
pdftitle={},
pdfproducer={pdfLaTeX},
pdfpagemode=None,
bookmarksopen=true,
bookmarksnumbered=true,
backref=page]{hyperref}

\usepackage{tikz}
\usetikzlibrary{arrows,calc,decorations.pathreplacing,decorations.markings,intersections,shapes.geometric,through,fit,shapes.symbols,positioning,decorations.pathmorphing}

\usepackage{braket}

\usepackage[amsmath,thmmarks,hyperref]{ntheorem}
\usepackage{cleveref}

\newcommand\nthalias[1]{\AddToHook{env/#1/begin}{\crefalias{lemma}{#1}}}

\nthalias{definition}
\nthalias{example}
\nthalias{examples}
\nthalias{remark}
\nthalias{remarks}
\nthalias{convention}
\nthalias{notation}
\nthalias{construction}
\nthalias{theoremN}
\nthalias{propositionN}
\nthalias{corollaryN}
\nthalias{lemma}
\nthalias{proposition}
\nthalias{corollary}
\nthalias{theorem}
\nthalias{conjecture}
\nthalias{question}
\nthalias{assumption}

\creflabelformat{enumi}{#2#1#3}

\crefname{section}{Section}{Sections}
\crefformat{section}{#2Section~#1#3} 
\Crefformat{section}{#2Section~#1#3} 

\crefname{subsection}{\S}{\S\S}
\AtBeginDocument{%
  \crefformat{subsection}{#2\S#1#3}%
  \Crefformat{subsection}{#2\S#1#3}% 
}

\crefname{subsubsection}{\S}{\S\S}
\AtBeginDocument{%
  \crefformat{subsubsection}{#2\S#1#3}%
  \Crefformat{subsubsection}{#2\S#1#3}% 
}

\theoremstyle{plain}

\newtheorem{lemma}{Lemma}[section]
\newtheorem{proposition}[lemma]{Proposition}
\newtheorem{corollary}[lemma]{Corollary}
\newtheorem{theorem}[lemma]{Theorem}

\theoremstyle{plain}
\theoremnumbering{Alph}

\theoremstyle{plain}
\theorembodyfont{\upshape}
\theoremsymbol{\ensuremath{\blacklozenge}}

\newtheorem{definition}[lemma]{Definition}

\newtheorem{remark}[lemma]{Remark}
\newtheorem{remarks}[lemma]{Remarks}

\crefname{definition}{definition}{definitions}
\crefformat{definition}{#2definition~#1#3} 
\Crefformat{definition}{#2Definition~#1#3} 

\crefname{ex}{example}{examples}
\crefformat{example}{#2example~#1#3} 
\Crefformat{example}{#2Example~#1#3} 

\crefname{exs}{example}{examples}
\crefformat{examples}{#2example~#1#3} 
\Crefformat{examples}{#2Example~#1#3} 

\crefname{remark}{remark}{remarks}
\crefformat{remark}{#2remark~#1#3} 
\Crefformat{remark}{#2Remark~#1#3} 

\crefname{remarks}{remark}{remarks}
\crefformat{remarks}{#2remark~#1#3} 
\Crefformat{remarks}{#2Remark~#1#3} 

\crefname{convention}{convention}{conventions}
\crefformat{convention}{#2convention~#1#3} 
\Crefformat{convention}{#2Convention~#1#3} 

\crefname{notation}{notation}{notations}
\crefformat{notation}{#2notation~#1#3} 
\Crefformat{notation}{#2Notation~#1#3} 

\crefname{table}{table}{tables}
\crefformat{table}{#2table~#1#3} 
\Crefformat{table}{#2Table~#1#3}

\crefname{lemma}{lemma}{lemmas}
\crefformat{lemma}{#2lemma~#1#3} 
\Crefformat{lemma}{#2Lemma~#1#3} 

\crefname{proposition}{proposition}{propositions}
\crefformat{proposition}{#2proposition~#1#3} 
\Crefformat{proposition}{#2Proposition~#1#3} 

\crefname{propositionN}{proposition}{propositions}
\crefformat{propositionN}{#2proposition~#1#3} 
\Crefformat{propositionN}{#2Proposition~#1#3} 

\crefname{corollary}{corollary}{corollaries}
\crefformat{corollary}{#2corollary~#1#3} 
\Crefformat{corollary}{#2Corollary~#1#3} 

\crefname{corollaryN}{corollary}{corollaries}
\crefformat{corollaryN}{#2corollary~#1#3} 
\Crefformat{corollaryN}{#2Corollary~#1#3} 

\crefname{theorem}{theorem}{theorems}
\crefformat{theorem}{#2theorem~#1#3} 
\Crefformat{theorem}{#2Theorem~#1#3} 

\crefname{theoremN}{theorem}{theorems}
\crefformat{theoremN}{#2theorem~#1#3} 
\Crefformat{theoremN}{#2Theorem~#1#3} 

\crefname{enumi}{}{}
\crefformat{enumi}{#2#1#3}
\Crefformat{enumi}{#2#1#3}

\crefname{assumption}{assumption}{Assumptions}
\crefformat{assumption}{#2assumption~#1#3} 
\Crefformat{assumption}{#2Assumption~#1#3} 

\crefname{construction}{construction}{Constructions}
\crefformat{construction}{#2construction~#1#3} 
\Crefformat{construction}{#2Construction~#1#3} 

\crefname{question}{question}{Questions}
\crefformat{question}{#2question~#1#3} 
\Crefformat{question}{#2Question~#1#3} 

\crefname{equation}{}{}
\crefformat{equation}{(#2#1#3)} 
\Crefformat{equation}{(#2#1#3)}

\numberwithin{equation}{section}

\theoremstyle{nonumberplain}
\theoremsymbol{\ensuremath{\blacksquare}}

\newtheorem{proof}{Proof}
\newcommand\pf[1]{\newtheorem{#1}{Proof of \Cref{#1}}}

\newcommand\bC{{\mathbb C}}

\newcommand\bS{{\mathbb S}}
\newcommand\bT{{\mathbb T}}

\newcommand\bZ{{\mathbb Z}}

\newcommand\cL{{\mathcal L}}

\newcommand\cT{{\mathcal T}}

\newcommand\cX{{\mathcal X}}

\newcommand\fg{{\mathfrak g}}

\newcommand\fsl{\mathfrak{sl}}

\DeclareMathOperator{\id}{id}

\DeclareMathOperator{\End}{\mathrm{End}}

\DeclareMathOperator{\spn}{\mathrm{span}}

\DeclareMathOperator{\ev}{\mathrm{ev}}
\DeclareMathOperator{\coev}{\mathrm{coev}}

\DeclareMathOperator{\diag}{diag}

\DeclareMathOperator{\ootimes}{\overline{\otimes}}
\DeclareMathOperator{\hotimes}{\widehat{\otimes}}
\DeclareMathOperator{\totimes}{\widetilde{\otimes}}

\DeclareMathOperator{\ospn}{\overline{\mathrm{span}}}

\DeclareMathOperator{\G}{G}
\DeclareMathOperator{\OR}{O}

\DeclareMathOperator{\U}{U}

\newcommand{\cat}[1]{\textsc{#1}}

\newcommand{\qedhere}{\mbox{}\hfill\ensuremath{\blacksquare}}

\newcommand{\xrightarrowdbl}[2][]{%
  \xrightarrow[#1]{#2}\mathrel{\mkern-14mu}\rightarrow
}

\title{Spans of quantum-inequality projections}
\author{Alexandru Chirvasitu}

\begin{document}

\date{}

\newcommand{\Addresses}{{% additional braces for segregating \footnotesize
  \bigskip
  \footnotesize

  \textsc{Department of Mathematics, University at Buffalo}
  \par\nopagebreak
  \textsc{Buffalo, NY 14260-2900, USA}  
  \par\nopagebreak
  \textit{E-mail address}: \texttt{achirvas@buffalo.edu}

  % % \medskip
  % % 
  % % \textsc{Department of Mathematics, INSTITUTION}
  % % \par\nopagebreak
  % % \textsc{ADDRESS}
  % % \par\nopagebreak
  % % \textit{E-mail address}: \texttt{??}
  % % 

}}

\maketitle

\begin{abstract}
  A hereditarily atomic von Neumann algebra $A$ is a $W^*$ product of matrix algebras, regarded as the underlying function algebra of a quantum set. Projections in $A\overline{\otimes}A^{\circ}$ are interpreted as quantum binary relations on $A$, with the supremum of all $p\otimes (1-p)$ representing quantum inequality. We prove that the symmetrized weak$^*$-closed linear span of all such quantum-inequality projections is precisely the symmetric summand of the joint kernel of multiplication and opposite multiplication, a result valid without the symmetrization qualification for plain matrix algebras. The proof exploits the symmetries of the spaces involved under the compact unitary group of $A$, and related results include a classification of those von Neumann algebras (hereditarily atomic or not) for which the unitary group operates jointly continuously with respect to the weak$^*$ topology.
\end{abstract}


\vspace{.5cm}

\noindent{MSC 2020: 17B10; 22E47; 18M05; 22D10; 46L05; 46L10; 16D25; 03G12
  
  % 17B10 Representations of Lie algebras and Lie superalgebras, algebraic theory (weights)
  % 22E47 Representations of Lie and real algebraic groups: algebraic methods (Verma modules, etc.)
  % 18M05 Monoidal categories, symmetric monoidal categories
  % 22D10 Unitary representations of locally compact groups
  % 46L05 General theory of $C^*$-algebras
  % 46L10 General theory of von Neumann algebras
  % 16D25 Ideals in associative algebras
  % 03G12 Quantum logic 
}

%\tableofcontents

%%%%%%%%%%%%%%%%%%%%%%%%%%%%%%%%
%%%%%%%%%%%%%%%%%%%%%%%%%%%%%%%%
\section*{Introduction}

The present note is an outgrowth of considerations revolving around the quantum set/relation theory developed and studied in \cite{zbMATH07287276,zbMATH07828321,zbMATH07828323}. The \emph{von Neumann algebras} \cite[\S I.9]{blk} most prominent in that formalism are the \emph{hereditarily atomic} ones: of the form
\begin{equation}\label{eq:mprod}
  A
  \cong
  \ell^{\infty}(\cX)
  :=
  \prod^{\cat{W}^*}_{i\in I}\cL(X_i)
  ,\quad
  \begin{aligned}
    \cX&:=\left\{\text{Hilbert spaces }X_i\right\},\quad \dim X_i<\infty,\\
    \cL(\bullet)&:=\text{bounded operators}\\    
  \end{aligned}  
\end{equation}
where the product's superscript denotes the product \cite[\S 3]{zbMATH03248454} in the category $\cat{W}^*$ of von Neumann (or \emph{$W^*$}-)algebras (uniformly-bounded tuples in the usual Cartesian product). \cite[Proposition 5.4]{zbMATH07287276} ensures that the term is justified: the von Neumann algebras \Cref{eq:mprod} are precisely those whose $W^*$-subalgebras (i.e. weak$^*$-closed $*$-subalgebras) are all \emph{atomic} in the sense \cite[Definition III.5.9]{tak1} that non-zero \emph{projections} (as usual \cite[Definition I.2.4.1]{blk}, self-adjoint idempotents) dominate non-zero \emph{minimal} projections.

A gadget such as \Cref{eq:mprod} is to be regarded as an algebra $\ell^{\infty}(\cX)$ of functions on a ``quantum set'' $\cX$ \cite[Definitions 2.1 and 5.2]{zbMATH07287276} (hence the $\ell^{\infty}$ notation). Projections
\begin{equation*}
  p\in A\ootimes A^{\circ}=\ell^{\infty}(\cX)\ootimes \ell^{\infty}(\cX)^{\circ}
  ,\quad
  \begin{aligned}
    A^{\circ}&:=\text{opposite algebra}\\
    \ootimes&:=\text{$W^*$ tensor product \cite[Definition IV.5.1]{tak1}}
  \end{aligned}
\end{equation*}
are interpreted as quantum binary relations on the quantum set $\cX$ attached to $A$. An important role in the resulting predicate logic is played by the \emph{equality} relation $\delta_A\in A\ootimes A^{\circ}$, defined in \cite[\S 1.1]{zbMATH07287276} as the largest projection orthogonal to all $p\otimes (1-p)$ for all projections $p\in A$. The title's \emph{quantum-inequality projections}, then, are precisely the $p\otimes (1-p)$. 

It is in this context that it becomes natural to assess just how large such spaces of projections $p\otimes (1-p)$ are. Given the bijective correspondence \cite[Theorem III.2.7]{tak1}
\begin{equation*}
  \left\{\text{projections in $N$}\right\}=:P(N)\ni p
  \xmapsto{\quad}
  Np
\end{equation*}
between projections and weak$^*$-closed left ideals of von Neumann algebras, one version of the question concerns said ideals. The $W^*$-flavored problem has an algebraic counterpart, the role of \Cref{eq:mprod} being played by the \emph{full} product \cite[Definition 5.1]{zbMATH07287276}
\begin{equation}\label{eq:prod}
  \ell(\cX):=\prod_i \cL(X_i)
  \quad
  \left(\text{Cartesian product}\right),
\end{equation}
equipped with the (\emph{locally convex} \cite[\S\S 18.1 and 18.3(5)]{k_tvs-1}) product topology. The appropriate replacement for $\ootimes$ will be $\hotimes$, the \emph{projective} \cite[\S 41.2(4)]{k_tvs-2} locally convex tensor product. Following \cite[Definition II.2.1]{tak1}, and in order to avoid collisions with \emph{strong$^*$} topologies (\cite[Definition II.2.3]{tak1}, \cite[\S I.3.1.6]{blk}), where the $*$ plays a different role, we henceforth refer to the weak$^*$ topology on a von Neumann algebra as \emph{$\sigma$-weak}.

\begin{theorem}\label{th:p1p.id}
  \begin{enumerate}[(1),wide]
  \item\label{item:th:p1p.id:vn} For a hereditarily-atomic von Neumann algebra \Cref{eq:mprod} the left (right) $\sigma$-weakly closed ideal of $\ell^{\infty}(\cX)\ootimes \ell^{\infty}(\cX)^{\circ}$ generated by all
    \begin{equation*}
      m_p:=p\otimes(1-p),\quad p\in P(\ell^{\infty}(\cX))
    \end{equation*}
    is the kernel of the multiplication map $\ell^{\infty}(\cX)\ootimes \ell^{\infty}(\cX)^{\circ}\xrightarrow{\mu} \ell^{\infty}(\cX)$ (respectively the kernel of the opposite multiplication $\mu^{\circ}$).

  \item\label{item:th:p1p.id:prod} Similarly, the closed left (right) ideal of $\ell(\cX)\hotimes \ell(\cX)^{\circ}$ generated by all $m_p$ is the kernel of $\mu$ (respectively $\mu^{\circ}$). 
  \end{enumerate}
\end{theorem}

This can be seen directly, by fairly standard von-Neumann-algebra machinery. The left ideal in \Cref{th:p1p.id}\Cref{item:th:p1p.id:vn} is nothing but $\left(A\ootimes A^{\circ}\right)(1-\delta_A)$, and \cite[Proposition 1.4]{ck_diskac_pre} (say) is not difficult to translate into \Cref{th:p1p.id}\Cref{item:th:p1p.id:vn}. The theorem is also a consequence, however, of determining what the appropriately closed \emph{linear span} of all $p\otimes (1-p)$ looks like. One (perhaps less guessable) answer reads as follows, with signed subscripts $\pm$ respectively denoting the $\pm 1$-eigenspaces of the flip operator $a\otimes b\xmapsto{\sigma} b\otimes a$.

\begin{theorem}\label{th:p1p.spn}
  \begin{enumerate}[(1),wide]
  \item\label{item:th:p1p.spn:vn} Let $A$ be a hereditarily-atomic von Neumann algebra \Cref{eq:mprod} and set
    \begin{equation}\label{eq:spn.mp}
      \cat{INQ}(A):=\ospn \left\{m_p:=p\otimes(1-p)\ :\ p\in P(A)\right\}
      \le
      A\ootimes A
    \end{equation}
    ($\sigma$-weak closure).
    
    The space
    \begin{equation*}
      \cat{INQ}(A)_+
      =
      \ospn \left\{m_p+\sigma m_p\ :\ p\in P(A)\right\}
    \end{equation*}
    equals
    \begin{equation*}
      \left(\ker\mu\bigcap\ker\mu^{\circ}\right)_+
      ,\quad\text{where}\quad
      \begin{aligned}
        \mu
        &:=\text{multiplication}\\
        \mu^{\circ}
        &:=\text{opposite multiplication}.
      \end{aligned}      
    \end{equation*}

  \item\label{item:th:p1p.spn:prod} The analogous result holds for products $A$ as in \Cref{eq:prod}, with closures in the product topology and $\hotimes$ in place of $\ootimes$. 
  \end{enumerate}
\end{theorem}

The \emph{anti}symmetric counterpart to the statement does not hold in full generality, but does for matrix algebras (\Cref{pr:p1p.spn.mat}):
\begin{equation}\label{eq:mn.mu.mu}
  \forall n\in \bZ_{\ge 1}
  \quad:\quad
  \cat{INQ}(M_n)
  =
  \left(\ker\mu\bigcap \ker\mu^{\circ}\right). 
\end{equation}

One perspective on such results as \Cref{eq:mn.mu.mu} and \Cref{th:p1p.spn} casts them as aids to elucidating and strengthening the relationship between the two types of structure involved:
\begin{itemize}[wide]
\item On the one hand, linearity is natural and ubiquitous in quantum mechanics, where it underlies both \emph{superposition} \cite[\S 1.2]{wev_mq} and the formation of \emph{mixed states} (\cite[\S 2-2, pp.79-81]{mack_qm_1963}, \cite[\S 2.3, post Fig.2.2]{ball_qm_2e_2015}) from pure.
 
\item On the other, multiplicative structure is notoriously more problematic and mysterious, with operator commutativity being a manifestation of the mutual compatibility of the corresponding observables and non-commutativity underpinning the celebrated \emph{uncertainty principle} \cite[(8.31) and (8.33)]{ball_qm_2e_2015}. 
\end{itemize}

In the context of the predicate calculus afforded by quantum sets and relations, \Cref{th:p1p.spn} and its matrix-algebra precursor \Cref{eq:mn.mu.mu} serve to further highlight the intimate connection between the linear and multiplicative aspects of the objects involved. 

The proof of \Cref{th:p1p.spn} is representation-theoretic in nature (as is that of \Cref{eq:mn.mu.mu}), leveraging the symmetry of all spaces in sight under both
\begin{itemize}[wide]
\item the flip map interchanging tensorands;
\item and the unitary group $\U(A)$, compact under its product topology (coincident on $\U(A)$ with any of the weaker-than-norm topologies one typically \cite[\S 2.1.7]{ped-aut} equips a von Neumann algebra with). 
\end{itemize}

Even the finite-dimensional (or indeed, single-matrix-algebra: \Cref{pr:p1p.spn.mat}) case requires an analysis of how $\cat{INQ}(A)$ decomposes under the conjugation action by $\U(A)$. Additionally, the infinite-dimensional version of the discussion prompts a number of considerations and side-notes on the applicability of the usual machinery \cite[Chapters 3 and 4]{hm5} of compact-group actions on locally convex topological vector spaces. This spills over into \Cref{se:jcont.wast}; its contents, perhaps of some independent interest, include:
\begin{itemize}[wide]
\item The characterization in \Cref{pr:cont.act.virt.cls} of those quantum sets for which the conjugation action of $\U(A)$ on some (equivalently, all) tensor powers $A^{\ootimes n}$, $A:=\ell^{\infty}(\cX)$ is jointly continuous for the weak topology on $\U(A)$ and $\sigma$-weak topology on $A$: they are precisely what we call the \emph{virtually classical} quantum sets, i.e. those $\cX=\left\{X_i\right\}$ with at most finitely many $(\ge 2)$-dimensional $X_i$. 

\item The analogous characterization (\Cref{th:cofin.ablnz}) of those von Neumann algebras $A$ for which the unitary group acts jointly continuously on any (all) of the tensor powers $A^{\ootimes n}$ as precisely the ``finite-dimensional-by-abelian'' ones
  \begin{equation*}
    A\cong F\times A_{ab}
    ,\quad
    \dim F<\infty
    \quad\text{and}\quad
    A_{ab}\text{ abelian}. 
  \end{equation*}

\item a discussion of automatic joint continuity under appropriate conditions (\Cref{le:fin.isot.jcont}) and its failure otherwise (\hyperref[res:joint.cont]{Remarks~\ref*{res:joint.cont}}). 
\end{itemize}

% % OLD: WRONG
% % 
% % \begin{theorem}\label{th:p1p.spn}
% %   For a finite-dimensional $C^*$-algebra $A$ the following subspaces of $A\otimes A$ coincide.
% %   \begin{enumerate}[(a),wide]
% %     
% %   \item\label{item:th:p1p.spn:spn.mp} the linear span
% %     \begin{equation}\label{eq:spn.mp}
% %       \spn\left\{m_p:=p\otimes(1-p)\ :\ p\in P(A)\right\};
% %     \end{equation}
% % 
% %   \item\label{item:th:p1p.spn:mumu} the joint kernel
% %     \begin{equation*}
% %       \ker\mu\bigcap\ker\mu^{\circ}
% %       ,\quad
% %       \begin{aligned}
% %         \mu
% %         &:=\text{multiplication}\\
% %         \mu^{\circ}
% %         &:=\text{opposite multiplication}.
% %       \end{aligned}      
% %     \end{equation*}
% %   \end{enumerate}
% % \end{theorem}

%%%%%%%%%%%%%%%%%%%%%%%%%%%%%%%%
\subsection*{Acknowledgments}

I am grateful for numerous illuminating exchanges with A. Kornell on and around quantum sets, relations, functions and like structures. 

% % %%%%%%%%%%%%%%%%%%%%%%%%%%%%%%%%
% % %%%%%%%%%%%%%%%%%%%%%%%%%%%%%%%%
% % \section{Preliminaries}\label{se:prel}
% %

%%%%%%%%%%%%%%%%%%%%%%%%%%%%%%%%
%%%%%%%%%%%%%%%%%%%%%%%%%%%%%%%%
\section{Linear spans of projection tensors}\label{se:spn.proj}

$\U(A)$ will denote the unitary group of a $C^*$-algebra $A$, with the abbreviations
\begin{equation*}
  \U(n):=\U(M_n)
  \quad\text{for}\quad
  M_n:=M_n(\bC). 
\end{equation*}
We lay some of the groundwork for handling the matrix case $A\cong M_n$ of \Cref{th:p1p.spn}. Set $V:=\bC^n$, $n\in \bZ_{\ge 1}$ with a fixed inner product $\braket{-\mid -}$ throughout, $(e_i)_{i=1}^n$ being an orthonormal basis. The maps
\begin{equation*}
  V^*\otimes V
  \xrightarrow[\quad\left(\text{evaluation}\right)\quad]{\quad\ev\quad}
  \bC
  \quad\text{and}\quad
  \bC
  \ni 1
  \xmapsto[\quad\left(\text{coevaluation}\right)\quad]{\quad\coev\quad}
  \sum_i e_i\otimes e_i^*
  \in
  V\otimes V^*
\end{equation*}
make the algebraic dual $V^*$ into a \emph{left dual} \cite[Definition 2.10.1]{egno} of $V^*$ in familiar monoidal-category terminology.

There is an identification
\begin{equation*}
  M_n:=M_n(\bC)\cong \End(V)\cong V\otimes V^*,
\end{equation*}
with
\begin{equation*}
  M_n\otimes M_n\cong (V\otimes V^*)\otimes (V\otimes V^*)
  \xrightarrow[\quad\mu:=\text{multiplication map}\quad]{\quad\id_V\otimes \ev\otimes \id_{V^*}\quad}
  V\otimes V^*
  \cong
  M_n.
\end{equation*}
$\fsl_n\subset M_n$ is the Lie algebra of traceless $n\times n$ matrices. The matrix-algebra version of \Cref{th:p1p.spn} reads as follows.

\begin{proposition}\label{pr:p1p.spn.mat}
  Let $n\in \bZ_{\ge 1}$. The following subspaces of $M_n\otimes M_n$ coincide.
  \begin{enumerate}[(a),wide]
    
  \item\label{item:pr:p1p.spn.mat:spn.mp} the span
    \begin{equation}\label{eq:spn.mp.mat}
      \spn\left\{m_p:=p\otimes(1-p)\ :\ p\in P(M_n)\right\};
    \end{equation}
    
  \item\label{item:pr:p1p.spn.mat:mumu} the joint kernel
    \begin{equation}\label{eq:mumu.mat}
      \ker\mu\bigcap\ker\mu^{\circ}
      \xrightarrow[\quad\cong\quad]{\quad\text{as $\U(n)$-representations}\quad}
      \fsl_n^{\otimes 2}.
    \end{equation}
  \end{enumerate}  
\end{proposition}
\begin{proof}
  That \Cref{eq:spn.mp.mat} $\subseteq$ \Cref{eq:mumu.mat} is self-evident, so we focus on the opposite inclusion assuming $n\ge 2$ (for the case $n=1$ is trivial). 
  
  \begin{enumerate}[(I),wide]
  \item\label{item:pr:p1p.spn.mat:pfiso} \textbf{: The isomorphism claimed in \Cref{eq:mumu.mat}.} Observe that
    \begin{equation*}
      \begin{aligned}
        \ker\mu
        &=\ker\left(V\otimes V^*\otimes V\otimes V^*\xrightarrow{\ev_{23}}V\otimes V^*\right)        
          \quad\text{and}\quad\\
        \ker\mu^{\circ}
        &=\ker\left(V\otimes V^*\otimes V\otimes V^*\xrightarrow{\ev_{14}}V^*\otimes V\right),
      \end{aligned}
    \end{equation*}
    with $\ev_{ij}$ denoting evaluation of the $i^{th}$ tensorand against the $j^{th}$. Cycling the tensorands so that the first becomes the last and writing $\fg:=\fsl_n$, the intersection of the two kernels will be (isomorphic to)
    \begin{equation*}
      \left(\ker\ev\right)^{\otimes 2}
      \cong
      \fg^{\otimes 2}
      \subset
      M_n^{\otimes 2}
      \cong
      \left(V\otimes V^*\right)^{\otimes 2}
      \cong
      \left(V^*\otimes V\right)^{\otimes 2},
    \end{equation*}
    all maps in sight being equivariant for the tensor various products of copies of the \emph{adjoint action} \cite[\S 8.1]{fh_rep-th} of $\U(n)$ on $\fg$ and its (\emph{co}adjoint) dual. 

  \item \textbf{: General remarks.} Note first that \Cref{eq:spn.mp.mat} is an $\U(n)$-subrepresentation of $\fg^{\otimes 2}$, so we can work $\U(n)$-equivariantly throughout.

    We remind the reader how $\fg^{\otimes 2}$ decomposes into irreducible subrepresentations; this is worked out (say) in \cite[\S 1]{2212.13787v1} for $n\ge 4$, and the picture can easily be completed to the two lower cases. We always have the decomposition
    \begin{equation*}
      \fg^{\otimes 2}
      =
      S^2\fg \oplus \wedge^2\fg
      \quad
      \left(\text{symmetric and exterior squares respectively}\right),
    \end{equation*}
    and the two summands further decompose as follows. 

    \begin{itemize}[wide]
    \item If $n\in \bZ_{\ge 4}$ then
      \begin{equation*}
        \begin{aligned}
          S^2\fg
          &\cong
            \fg^{(2)}\oplus \fg^{(1^2)}\oplus \fg\oplus \mathbf{1}
            ,\quad
            \mathbf{1}:=\text{trivial representation}\\
          \wedge^2 \fg
          &\cong
            \fg^{(1^2,2)}\oplus \fg^{(2,1^2)}\oplus \fg,
        \end{aligned}        
      \end{equation*}
      where the \emph{highest weights} attached \cite[Theorem 14.18]{fh_rep-th} to the irreducible $\U(n)$-representations $\fg^{(2)}$, $\fg^{(1^2)}$, $\fg^{(1^2,2)}$ and $\fg^{(2,1^2)}$ with respect to the maximal torus
      \begin{equation*}
        \bT:=\left\{\diag(z_i)_{i=1}^n\ :\ z_i\in \bS^1\subset \bC\right\}\subset \U(n)
      \end{equation*}
      are
      \begin{equation*}
        \begin{aligned}
          \fg^{(2)}
          &\xrightarrow{\qquad}
            z_1^2 z_n^{-2}\\
          \fg^{(1^2)}
          &\xrightarrow{\qquad}
            z_1z_2 z_{n-1}^{-1}z_n^{-1}\\
          \fg^{(1^2,2)}
          &\xrightarrow{\qquad}
            z_1^2 z_{n-1}^{-1} z_n^{-1}\\
          \fg^{(2,1^2)}
          &\xrightarrow{\qquad}
            z_1 z_2 z_n^{-2}.\\
        \end{aligned}
      \end{equation*}

    \item If $n=3$ then the $\fg^{(1^2)}$ summand is absent. 

    \item If $n=2$ then 
      \begin{equation*}
        S^2\fg
        \cong
        \fg^{(2)}\oplus \mathbf{1}
        \quad\text{and}\quad
        \wedge^2 \fg
        \cong
        \fg.
      \end{equation*}
    \end{itemize}

  \item \textbf{: The ``bulk'' summands $\fg^{\bullet}$.} In terms of the usual matrix units $e_{ij}\in M_n$, $1\le i,j\le n$, the following elements $x\in M_n^{\otimes 2}$ are highest-weight vectors for the summands denoted by superscripts as follows:
    \begin{equation*}
      \begin{aligned}
        \fg^{(2)}
        &\quad:\quad e_{1n}\otimes e_{1n}\\
        \fg^{(1^2)}
        &\quad:\quad e_{1n}\otimes e_{2,n-1}+e_{2,n-1}\otimes e_{1n} + e_{1,n-1}\otimes e_{2,n}+e_{2,n} \otimes e_{1,n-1}\\
        g^{(1^2,2)}
        &\quad:\quad
          e_{1n}\otimes e_{1,n-1}-e_{1,n-1}\otimes e_{1n}
        \\
        g^{(2,1^2)}
        &\quad:\quad
          e_{1n}\otimes e_{2n}-e_{2n}\otimes e_{1n}.
      \end{aligned}      
    \end{equation*}
    In each case it is a simple remark that $\braket{x\mid m_p}$, $p\in P(M_n)$ cannot all vanish, where $\braket{-\mid-}$ denotes the $\U(n)$-invariant inner product
    \begin{equation}\label{eq:abcd.inner.prod}
      \braket{a\otimes b\mid c\otimes d}
      :=
      \tau(a^*c)\tau(b^*d)
      ,\quad
      \tau:=\text{normalized trace}
    \end{equation}
    on $M_n^{\otimes 2}$. It follows that the space \Cref{eq:spn.mp.mat} cannot avoid any of the $\fg^{\bullet}$ summands. 

  \item \textbf{: The trivial representation.} Simply observe that
    \begin{equation*}
      \mathrm{av}_p
      :=
      \int_{\U(n)}g m_p\ \mathrm{d}\mu_{\U(n)}(g)
      \in
      \text{space \Cref{eq:spn.mp.mat}}
      ,\quad
      \mu_{\U(n)}:=\text{\emph{Haar measure} \cite[\S I.5]{btd_lie_1995} on $\U(n)$}
    \end{equation*}
    is $\U(n)$-fixed and non-zero whenever the projections $p$ and $1-p$ are both non-zero, for in that case
    \begin{equation*}
      \braket{\mathrm{av}_p\mid 1}=\tau(p)\tau(1-p)\ne 0.
    \end{equation*}

  \item \textbf{: The $\fg$ summands.} Recall from step \Cref{item:pr:p1p.spn.mat:pfiso} above that we are realizing $\fg^{\otimes 2}$ as a subspace of $M_n^{\otimes 2}\cong (V\otimes V^*)^{\otimes 2}$ by embedding one $\fg$ in the inner $V^*\otimes V\cong V\otimes V^*$ and the other in the outer $V\otimes V^*$. With that convention, applying the trace $n\tau$ to the right-hand $M_n$ in the ambient space $M_n^{\otimes 2}$ gives $\U(n)$-equivariant maps to $M_n$ when restricted to either $S^2 \fg$ or $\wedge^2 \fg$. Next note that
    \begin{equation*}
      \wedge^2\fg
      \ni
      p\otimes (1-p)-(1-p)\otimes p
      \xmapsto{\quad\id\otimes n\tau\quad}
      n\tau(1-p)p-n\tau(p)(1-p)\in \fg
    \end{equation*}
    is always non-scalar somewhere, while 
    \begin{equation*}
      S^2\fg
      \ni
      p\otimes (1-p)+(1-p)\otimes p
      \xmapsto{\quad\id\otimes n\tau\quad}
      n\tau(1-p)p+n\tau(p)(1-p)\in \fg
    \end{equation*}
    takes only scalar values precisely when $n=2$, in which case there is no $\fg$ summand in $S^2\fg$ to begin with. This suffices to conclude that
    \begin{equation*}
      \fg\le \spn\{m_p\}\cap S^2\fg
      \ \left(n\ge 3\right)
      \quad\text{and}\quad
      \fg\le \spn\{m_p\}\cap \wedge^2\fg,
    \end{equation*}
    finishing the proof.     
  \end{enumerate}
\end{proof}

The statement of the \Cref{cor:id.p1p}, recording a consequence of \Cref{pr:p1p.spn.mat}, follows standard practice (e.g. \cite[\S 10.1]{pierce_assoc}) in writing $A^e:=A\otimes A^{\circ}$ for the \emph{enveloping algebra} of an algebra $A$.

\begin{corollary}\label{cor:id.p1p}
  The left (right) ideal of $M_n^e$ generated by all $p\otimes (1-p)$ is precisely $\ker\mu$ (respectively $\ker\mu^{\circ}$).  
\end{corollary}
\begin{proof}
  The two statements are mirror images of each other, so we focus on the first. Plainly, not every element $ap\otimes (1-p)b$ in the left ideal generated by \Cref{eq:spn.mp} belongs to $\ker\mu^{\circ}$, so that left ideal is larger. The conclusion follows from the simplicity of 
  \begin{equation*}
    \ker\mu/\left(\ker\mu\cap\ker\mu^{\circ}\right)
    \quad\cong\quad
    \left(V\otimes \fsl_n\otimes V^*\right)/\fsl_n^{\otimes 2}
    \quad\cong\quad
    \fsl_n
  \end{equation*}
  as a $\U(n)$-representation.
\end{proof}

\pf{th:p1p.id}
\begin{th:p1p.id}
  There is no loss in handling only the left-handed versions of the claims. Note first that the multiplication maps
  \begin{equation*}
    \ell^{\infty}(\cX)\ootimes \ell^{\infty}(\cX)^{\circ}
    \xrightarrow{\quad\mu\quad}
    \ell^{\infty}(\cX)
    \quad\text{and}\quad
    \ell(\cX)\hotimes \ell(\cX)^{\circ}
    \xrightarrow{\quad\mu\quad}
    \ell(\cX)
  \end{equation*}
  pose no continuity issues: in the latter case by the universality property of the projective tensor product \cite[\S 41.3(1)]{k_tvs-2}, and in the former case because \cite[Propositions 8.4 and 8.6]{zbMATH03248454}
  \begin{equation*}
    \ell^{\infty}(\cX)\ootimes \ell^{\infty}(\cX)^{\circ}
    \cong
    \prod_{i,j}^{\cat{W}^*}\cL(X_i)\ootimes \cL(X_j)^{\circ}
  \end{equation*}
  and the usual multiplication $M_n\otimes M_n^{\circ}\to M_n$ is \emph{multiplicatively contractive} in the sense of \cite[post (9.1.4)]{er_os_2000}. We treat the two cases uniformly, the product symbol doing double duty ($W^*$ or Cartesian), as well as the closure operator ($\sigma$-weak and product topology respectively) and the symbol $A$ (which stands for either $\ell^{\infty}$ or $\ell$).

  Denoting by $1_i\in A_i:=\cL(J_i)$ the respective unit, we have 
  \begin{equation*}
    \begin{aligned}
      \overline{\sum_{p\in p\in \ell}}A m_p
      &=
        \prod_{i,j}\sum_{\substack{p_i\in P(A_i)\\p_j\in P(A_j)}} \left(A_i\otimes A_j^{\circ}\right)(p_i\otimes (1-p_j))\\
      &=
        \left(
        \prod_{i\ne j}
        A_i\otimes A_j^{\circ}
        \right)
        \times
        \prod_i \sum_{p\in P(A_i)}\left(A_i\otimes A_i^{\circ}\right)m_p,
    \end{aligned}
  \end{equation*}
  reducing the claim(s) to the matrix case covered by \Cref{cor:id.p1p}. 
\end{th:p1p.id}

We next turn to \Cref{th:p1p.spn}, focusing on the von Neumann version of throughout (the full-product case being entirely parallel). All products of $W^*$-algebras appearing below, consequently, will be the categorically-appropriate ones for $\cat{W}^*$ (as in \Cref{eq:mprod}).

The general case to consider is
\begin{equation*}
  \begin{aligned}
    A
    &\cong
      \prod_{i\in I} A_i
      ,\quad
      A_i\cong M_{n_i}
      ,\quad
      n_i\in \bZ_{\ge 1},\\
    \U:=\U(A)
    &\cong
      \prod_{i\in I} \U_i
      ,\quad
      \U_i:=\U(A_i)\cong \U(n_i).
  \end{aligned}    
\end{equation*}

The following simple remark will allow us to assume the $A$ of \Cref{th:p1p.spn}\Cref{item:th:p1p.spn:prod} finite-dimensional for the most part. 

\begin{lemma}\label{le:fin.dim.enough}
  \Cref{th:p1p.spn}\Cref{item:th:p1p.spn:prod} is equivalent to its counterpart for finite-dimensional $C^*$-algebras. 
\end{lemma}
\begin{proof}
  This is immediate: the product-topology closure of a linear subspace $W\le A$ is precisely the set of elements $(a_i)\in A=\prod_i A_i$ with
  \begin{equation*}
    (a_j)_{j\in F}\in \pi_F(W)
    ,\quad
    \forall\text{ projection }
    A=\prod_{i\in I}A_i
    \xrightarrowdbl{\quad\pi_F\quad}
    A_F:=\prod_{j\in \text{ finite }F\subseteq I}A_j
  \end{equation*}
  (for arbitrary finite subsets $F\subseteq I$). The same goes for subspaces of
  \begin{equation*}
    A\hotimes A\cong \prod_{i,j\in I}A_i\otimes A_j
  \end{equation*}
  (see the proof of \Cref{th:p1p.id}). 
\end{proof}

We equip $\U$ with its product (hence compact Hausdorff) topology, and take for granted some of the standard background on compact-group representation theory on (always Hausdorff) locally convex topological vector spaces (as covered variously by \cite[Chapters 3 and 4]{hm5}, \cite[Chapter 7]{rob}, and like sources). The phrase \emph{representation} in the sequel, for compact $\G$ and locally convex $E$, refers to \emph{separately continuous} actions
\begin{equation*}
  \G\times E\xrightarrow{\quad\triangleright\quad}E
  \quad
  \left(\text{sometimes }\triangleright:\G \circlearrowright E\right)
\end{equation*}
in the sense of being continuous in each variable when the other is fixed. This often but not always entails \emph{joint} continuity, e.g. \cite[p.VIII.9, Proposition 1]{bourb_int_en_7-9} when $E$ is \emph{barreled} in the sense of \cite[\S 21.2]{k_tvs-1} (cf. \cite[discussion in \S 2, p.13]{rob} and \cite[p.VIII.9, Definition 1]{bourb_int_en_7-9}).

For compact $\G$ we write $\widehat{\G}$ for the set of (isomorphism classes of) \emph{irreducible} representations. These are automatically finite-dimensional \cite[Corollary 7.9]{rob} if assumed jointly-continuous and \emph{quasi-complete} \cite[post \S 18.4(3)]{k_tvs-1}. For a $\G$-representation $\G\circlearrowright E$ on a locally convex space $E$ we denote by $E_{\rho}$, $\rho\in \widehat{\G}$ the \emph{$\rho$-isotypic component} \cite[Definition 4.21]{hm5} of the representation: the largest $\U$-invariant subspace of $E$ decomposing as a sum of copies of $\rho$. These do exist and come packaged with idempotents
\begin{equation*}
  E
  \xrightarrowdbl{\quad P_{\rho}\quad}
  E_{\rho}
  \le
  E
\end{equation*}
as soon as the locally convex topology of $E$ is complete in the appropriate sense (quasi-completeness suffices \cite[Table 3.1 and Theorem 3.36]{hm5}). In order to take full advantage of the $\U$-symmetry of all spaces involved in the sequel, we need the following remark.

\begin{lemma}\label{le:cont.proj}
  Let $\cX$ be a quantum set \Cref{eq:mprod} and set $\ell^{\bullet}:=\ell^{\bullet}(\cX)$ for $\bullet\in \left\{\text{blank},\ \infty\right\}$.
  \begin{enumerate}[(1),wide]
  \item\label{item:le:cont.proj:l} The actions of $\U=\U\left(\ell\right)$ on the tensor powers $\ell^{\hotimes n}$ with the product topology are jointly continuous. 

  \item\label{item:le:cont.proj:linfty} The actions of $\U$ on the tensor powers $\left(\ell^{\infty}\right)^{\ootimes n}$ with the $\sigma$-weak topology are separately continuous.

  \item\label{item:le:cont.proj:cont.proj} In both cases $E\in \left\{\ell^{\hotimes n},\ \left(\ell^{\infty}\right)^{\ootimes n}\right\}$ the idempotents $P_\rho$ onto the $\rho$-isotypic component $E_{\rho}$, $\rho\in\widehat{\U}$ are continuous. 
  \end{enumerate}
\end{lemma}
\begin{proof}
  Observe first that the spaces $\ell^{\bullet}$ are indeed quasi-complete in their respective topologies, so the $\ell^{\bullet}$-integration theory discussed in \cite[Proposition 3.30]{hm5} does indeed apply:
  \begin{itemize}[wide]
  \item $\ell$ is even complete by \cite[post \S 15.4(7)]{k_tvs-1}, being a product of finite-dimensional (hence complete) topological vector spaces;
    
  \item while $\ell^{\infty}$, while not $\sigma$-weakly-complete if infinite-dimensional (e.g. by \cite[\S 39.6(7)]{k_tvs-2}, because it is not the \emph{full} algebraic dual of its predual), is nevertheless \emph{quasi}-complete by \cite[\S 39.6(5)]{k_tvs-2}.
  \end{itemize}

  Recall that (as effectively noted in the proof of \Cref{th:p1p.id} for $n=2$) 
  \begin{equation}\label{eq:l.tens.n.decomp}
    \ell^{\hotimes n}
    \cong
    \prod_{\mathbf{i}\in I^n}A_{\mathbf{i}}
    \quad\text{and}\quad
    \left(\ell^{\infty}\right)^{\ootimes n}
    \cong
    \prod^{\cat{W}^*}_{\mathbf{i}\in I^n}A_{\mathbf{i}}
    ,\quad
    A_{\mathbf{i}}:=A_{i_1}\otimes\cdots \otimes A_{i_n}
    \text{ for }
    \mathbf{i}=\left(i_j\right)_j\in I^n.
  \end{equation}
  We address the various claims in turn.

  \begin{enumerate}[label={}, wide]
  \item \textbf{\Cref{item:le:cont.proj:cont.proj}:} In both cases the isotypic projections $P_{\rho}$ are products (Cartesian or $W^*$) of projections on the individual factors of \Cref{eq:l.tens.n.decomp}. Continuity follows:
    \begin{itemize}[wide]
    \item in the Cartesian case categorically ($P_{\rho} being $the product in the category of locally convex spaces of a family of morphisms);
    \item and in the $W^*$ case (by \cite[Theorem 8.10.3]{nb_tvs}, say) because $P_{\rho}$ is the dual of an endomorphism of the predual
      \begin{equation}\label{eq:l1.predual}
        \ell^1\left(A_{\mathbf{i}*}\right)_{\mathbf{i}}
        :=
        \left\{(a_{\mathbf{i}})_{\mathbf{i}}\in \prod_{I^n} A_{\mathbf{i}*}\ :\ \sum_{\mathbf{i}} \|a_{\mathbf{i}}\|<\infty\right\},
      \end{equation}
      (\emph{$\ell^1$-sum} \cite[\S 1.1, post Proposition 7]{hlmsk_fa}), where lower $*$ indicate preduals (respectively identifiable with $A_{\mathbf{i}}$ via the pairing afforded by \Cref{eq:abcd.inner.prod}).
    \end{itemize}

  \item \textbf{\Cref{item:le:cont.proj:l} and \Cref{item:le:cont.proj:linfty}: separate continuity:} Separate continuity is easily checked for the $\U$-actions on
    \begin{itemize}
    \item the algebraic direct sum $\bigoplus_{\mathbf{i}\in I^n}A_{\mathbf{i}*}$ with its locally convex \emph{direct-sum topology} \cite[\S 18.5]{k_tvs-1};

    \item as well as on the $\ell^1$-sum \Cref{eq:l1.predual}
    \end{itemize}
    The \emph{weak duals} \cite[\S 20.2]{k_tvs-1} of these spaces being $\ell$ and $\ell^{\infty}$ respectively, separate continuity follows in both cases from \cite[p.VIII.11, Proposition 3(i)]{bourb_int_en_7-9}. This disposes of \Cref{item:le:cont.proj:linfty}. 
    
  \item \textbf{\Cref{item:le:cont.proj:l}:} As for the stronger \emph{joint} continuity claim in \Cref{item:le:cont.proj:l}, it follows from the already-cited \cite[p.VIII.9, Proposition 1]{bourb_int_en_7-9} and the fact \cite[\S 27.1(5)]{k_tvs-1} that $\ell$ is barreled in the sense of \cite[\S 21.2]{k_tvs-1}.
  \end{enumerate}
\end{proof}

\begin{remark}\label{re:linfty.not.joint.cont}
  The statement of \Cref{le:cont.proj}\Cref{item:le:cont.proj:linfty} can certainly \emph{not} (generally) be strengthened to joint continuity for infinite-dimensional hereditarily-atomic von Neumann algebras; in fact, those quantum sets for which that strengthening does obtain can be classified, per \Cref{def:qset.virt.cls} and \Cref{pr:cont.act.virt.cls} below.
\end{remark}

\Cref{re:linfty.not.joint.cont} notwithstanding, the continuity of the component projections ensured by \Cref{le:cont.proj}\Cref{item:le:cont.proj:cont.proj} does afford us most of the Peter-Weyl theory summarized in \cite[Theorems 3.51 and 4.22]{hm5}. In particular, for subspaces $W,W'\le \ell^{\bullet}$ closed in the appropriate ($\sigma$-weak or product) topology, we have
\begin{equation*}
  W\le W'
  \xLeftrightarrow{\quad}
  W_{\rho}\le W'_{\rho}
  ,\quad
  \forall
  \rho\in\widehat{\U}
  \quad
  \left(\text{similarly for arbitrary }\U_S:=\prod_{s\in S}\U_s,\ S\subseteq I\right).
\end{equation*}
The appeals to this general observation will be frequent but mostly tacit.

Whenever $A$ happens to be finite-dimensional, we assume fixed an (automatically $\U$-invariant) inner products $\braket{-\mid -}=\braket{-\mid -}_{\tau}$ on $A\otimes A$ defined as in \Cref{eq:abcd.inner.prod} for a faithful tracial state
\begin{equation}\label{eq:ftrc}
  A\xrightarrow{\quad \tau=\sum_i \alpha_i\tau_i\quad}\bC
  ,\quad
  \begin{aligned}
    \tau_i
    &:=\text{normalized trace on $A_i$}\\
    \alpha_i
    &> 0
      ,\quad
      \sum_i\alpha_i=1.     
  \end{aligned}
\end{equation}
The ``cross'' summands $A_i\otimes A_j$, $i\ne j$ are not addressed by \Cref{pr:p1p.spn.mat}; we focus on those first, noting the $\U$-invariant decomposition
\begin{equation}\label{eq:aiaj.dec}
  A_i\otimes A_j
  \cong
  \bC
  \oplus
  \left(
    \fg_{i}\otimes \bC
    \oplus
    \bC\otimes \fg_{j}
  \right)
  \oplus
  \fg_{i}\otimes \fg_{j}
  ,\quad
  \fg_{\bullet}:=\fsl_{n_{\bullet}}\le A_{\bullet}.
\end{equation}

The three types of summands will be handled separately. In addressing the symmetric portion of \Cref{th:p1p.spn}, $\pm$ subscripts denote the $\pm 1$-eigenspaces of
\begin{equation*}
  A\otimes A
  \ni
  a\otimes b
  \xmapsto{\quad\sigma\quad}
  b\otimes a
  \in
  A\otimes A
\end{equation*}
operating on whatever the space in question is (as in the Introduction); more generally, $(V\otimes W)_{\pm}$ stands for the (anti)symmetric summand of the smallest $\sigma$-invariant space generated by $V\otimes W$:
\begin{equation}\label{eq:vwpm}
  \left(V\otimes W\right)_{\pm}
  :=
  \begin{cases}
    \left(V\otimes V\right)_{\pm}
    &\text{if }V=W\\
    \left(V\otimes W\oplus W\otimes V\right)_{\pm}
    &\text{otherwise}.
  \end{cases}
\end{equation}
We will also write
\begin{equation}\label{eq:abpm}
  \left(a\otimes b\right)^{\pm}
  :=
  a\otimes b\pm b\otimes a.
\end{equation}

\begin{lemma}\label{le:cnst}
  Under the hypotheses of \Cref{th:p1p.spn} the symmetric scalar summand
  \begin{equation*}
    \bC
    =
    \bC\left(1_i\otimes 1_j\right)^+
    \le
    \left(A_i\otimes A_j\oplus A_j\otimes A_i\right)_+
  \end{equation*}
  is contained in the linear span of $m_p^+$, $p\in P(A)$ (so in particular also in \Cref{eq:spn.mp}). 
\end{lemma}
\begin{proof}
  \begin{enumerate}[(I),wide]
  \item\label{item:le:cnst:pf.fd} \textbf{: Finite-dimensional $A$.} Write
    \begin{equation}\label{eq:as}
      A_S:=\prod_{j\in S}A_j
      ,\quad
      \widehat{\bullet}
      :=
      \text{complement of $\bullet$ in $I$}.
    \end{equation}
    Choosing for $p\in P(A)$ the \emph{support projections} \cite[Proposition II.3.12 and Definition III.2.8]{tak1} $1_S\in A_S$ of $A_S$ for varying $S$, it follows that the elements $1_S\otimes 1_{\widehat{S}}$ all belong to \Cref{eq:spn.mp}.

    For subsets $S,T\subseteq I$, write
    \begin{equation*}
      1_{S\mid T}^{\pm}
      :=
      \left(1_S\otimes 1_T\right)^{\pm}
      \quad
      \left(\text{cf. \Cref{eq:abpm}}\right),
    \end{equation*}
    with singletons $\{i\}$ abbreviated to $i$. I claim that 
    \begin{equation}\label{eq:1ss.1ij}
      \spn \left\{1_{S\mid \widehat{S}}^{+}\ :\ \emptyset\ne S\subsetneq I\right\}
      =
      \spn \left\{1_{i\mid j}^{+}\ :\ i\ne j\in I\right\}.
    \end{equation}  
    The inclusion `$\subseteq$' is obvious; for its opposite, for $i\ne j\in I$ set
    \begin{equation*}
      \cat{IO}_{i\mid j}
      \ 
      \left(\text{for in/out}\right)
      \ 
      :=
      \sum_{
        \substack{
          S\subseteq I\\
          |S\cap \left\{i,j\right\}|=1
        }
      }1_{S\mid \widehat{S}}^{+}
      -
      \sum_{
        \substack{
          S\subseteq I\\
          |S\cap \left\{i,j\right\}|\in \{0,2\}
        }
      }1_{S\mid \widehat{S}}^{+}.
    \end{equation*}
    We have   
    \begin{equation*}
      \Braket{\cat{IO}_{i\mid j}\ \big|\ 1_{i'\mid j'}^{+}}
      =
      \delta_{i,i'}\cdot \delta_{j,j'}\cdot
      \left(\text{positive factor}\right),
    \end{equation*}
    so no non-trivial linear combination can be $\Braket{-\mid-}$-orthogonal to all $1_{S\mid \widehat{S}}^+$; consequently, \Cref{eq:1ss.1ij} holds. 
    
  \item\label{item:le:cnst:pf.gen} \textbf{: General case.} Simply observe that \Cref{eq:1ss.1ij} for triple products $\prod_{i=1}^3 M_{n_i}$ delivers its infinite-product counterpart (whether Cartesian or von Neumann) upon identifying the center $\bC^3\le \prod_{k=1}^3 M_{n_k}$ with the central subalgebra $\bC 1_i\oplus \bC 1_j\oplus \bC 1_{\widehat{\{i,j\}}}$ of $A$. Indeed, denoting the three minimal central projections of $\bC^3$ by $1_{\alpha,\beta,\gamma}$ and casting the aforementioned identification as a Boolean-algebra isomorphism
    \begin{equation*}
      2^{\{\alpha,\beta,\gamma\}}
      \xrightarrow[\quad\cong\quad]{\psi}
      \text{Boolean subalgebra of $2^I$ generated by }
      i,\ j,\ \widehat{\left\{i,j\right\}}
    \end{equation*}
    sending $\alpha$, $\beta$ and $\gamma$ to $i$, $j$ and $\widehat{\left\{i,j\right\}}$ respectively, we have
    \begin{equation*}
      1^+_{\alpha\mid \beta}
      =
      \sum_{F\subseteq \left\{\alpha,\beta,\gamma\right\}} c_F 1^+_{F\mid \widehat{F}}
      \quad
      \xRightarrow{\quad}
      \quad
      1^+_{i\mid j}
      =
      1^+_{\psi \alpha\mid \psi \beta}
      =
      \sum_{F\subseteq \left\{\alpha,\beta,\gamma\right\}} c_F 1^+_{\psi F\mid \psi\widehat{F}}.
    \end{equation*}
    This completes the proof.
  \end{enumerate}
\end{proof}

%\newpage

We now consider the ``bulk'' summand of \Cref{eq:aiaj.dec}, again in its symmetrized version. 

\begin{proposition}\label{pr:only.symm}
  Under the hypotheses of \Cref{th:p1p.spn}, for $i\ne j$ the intersection
  \begin{equation*}
    \text{\Cref{eq:spn.mp}}
    \cap
    \left(
      \left(
        \fg_{i}\otimes \fg_{j}
      \right)_+
      \le
      \left(
        A_i\otimes A_j
      \right)_+
    \right)
    \quad
    \left(\text{per \Cref{eq:vwpm}}\right)
  \end{equation*}
  consists precisely of the symmetric tensors
  \begin{equation*}
    \spn\left\{x+\sigma x\ :\ x\in \fg_{i}\otimes \fg_{j},\ \sigma:=\text{tensorand interchange}\right\}.
  \end{equation*}
\end{proposition}
\begin{proof}
  $\fg_{i}\otimes \fg_{j}$ is an irreducible representation of the quotient $\U\xrightarrowdbl{} \U_i\times \U_j$, appearing in $A\otimes A$ with multiplicity 2 (in $A_i\otimes A_j$ and $A_j\otimes A_i$). For that reason, we may as well assume $A$ finite-dimensional (the factors $A_k$, $k\ne i,j$ are irrelevant to the discussion).
  
  The symmetric and antisymmetric copies are generated as $\U_i\times \U_j$-representations by
  \begin{equation*}
    \left(a_i\otimes a_j\right)^{\pm}
    ,\quad
    0\ne a_{\bullet}\in \fg_{{\bullet}}
    \quad
    \left(\text{in the notation of \Cref{eq:abpm}}\right)
  \end{equation*}
  respectively, and it suffices to observe, writing $p_i\in P(A_i)$ for the respective component of $p\in P(A)$, that
  \begin{equation*}
    \begin{aligned}
      \Braket{m_p\ |\ \left(a_i\otimes a_j\right)^{\pm}}
      &=
        \tau(p_i a_i)\tau((1-p_j)a_j)
        \pm
        \tau((1-p_i) a_i)\tau(p_j a_j)\\
      &\xlongequal{\quad \tau(a_{\bullet})=0\quad}
        -
        \left(
        \tau(p_i a_i)\tau(p_ja_j)
        \pm
        \tau(p_i a_i)\tau(p_ja_j)
        \right);
    \end{aligned}    
  \end{equation*}
  this vanishes for \emph{all} choices of $p\in P(A)$ precisely when the sign is `$-$'.
\end{proof}

\newtheorem{item:th:p1p.spn:sym}{Proof of \Cref{th:p1p.spn}}
\begin{item:th:p1p.spn:sym}
  The inclusion
  \begin{equation*}
    \cat{INQ}(A)\le \ker\mu\bigcap\ker\mu^{\circ}
  \end{equation*}
  being obvious, we have to prove the leftmost inclusion in 
  \begin{equation*}
    \begin{aligned}
      \left(
      \ker\mu\bigcap\ker\mu^{\circ}
      \right)_+
      &\le
        {\spn}\left\{m_p+\sigma m_p\ |\ p\in P(A)\right\}\\
      &=
        {\spn}\left\{m_p\ |\ p\in P(A)\right\}_+
        \le
        {\spn}\left\{m_p\ |\ p\in P(A)\right\}.
    \end{aligned}
  \end{equation*}
  There are a few stages to the argument.

  \begin{enumerate}[(I),wide]
  \item\label{item:pr:p1p.spn.sym:pf.red} \textbf{: Reducing the problem to
      \begin{equation}\label{eq:aiajajai}
        \left(A_i\otimes A_j\oplus A_j\otimes A_i\right)_+
        \le
        {\spn}\left\{m_p+\sigma m_p\ |\ p\in P(A)\right\}
      \end{equation}
      for $i\ne j$.} Assume all such inclusions are in place (for arbitrary $i\ne j$) and denote the $A_i$-component of a projection $p\in P(A)$ by $p_i$. Fix $i$ and consider $p\in P(A)$ with $p_j=0$ for all $j\ne i$ (or: $p=p_i\in P(A_i)\subset P(A)$), so that
    \begin{equation*}
      p\otimes (1-p) = p_i\otimes (1-p_i) + p\otimes 1_{\widehat{i}}
      \in
      \left(A_i\otimes A_i\right)
      \oplus
      \left(A_i\otimes A_{\widehat{i}}\right)
    \end{equation*}
    in the notation of \Cref{eq:as}. Since we are assuming that
    \begin{equation*}
      \left(p\otimes 1_i\right)^+
      \in
      {\spn}\left\{m_p+\sigma m_p\ |\ p\in P(A)\right\},
    \end{equation*}
    this would prove that
    \begin{equation*}
      \bigoplus_i
      {\spn}\left\{m_p+\sigma m_p\ |\ p\in A_i\right\}
      \le
      \bigoplus_i \left(A_i\otimes A_i\right)_+
    \end{equation*}
    is contained in the right-hand side of \Cref{eq:aiajajai}; the conclusion would then follow from \Cref{pr:p1p.spn.mat}.
    
  \item\label{item:pr:p1p.spn.sym:sc.blk} \textbf{: \Cref{eq:aiajajai}: scalars and bulk.} That is, recalling the decompositions \Cref{eq:aiaj.dec}, the claim here is that
    \begin{equation*}
      \left(\bC\le \left(A_i\otimes A_j\oplus A_j\otimes A_i\right)_+\right)
      \ \text{and}\ 
      \left(\fg_{i}\otimes \fg_{j}\oplus \fg_{j}\otimes \fg_{i}\right)_+
      \quad\le\quad
      \text{right side of \Cref{eq:aiajajai}}.
    \end{equation*}
    This follows from \Cref{le:cnst} and \Cref{pr:only.symm}.    

    % % (the $\fg_{i}$-\emph{isotypic component} \cite[\S 6.2.3]{proc_lie} of $\U$ in $A\otimes \bC\oplus \bC\otimes A$)
    % % 
    
  \item\label{item:pr:p1p.spn.sym:rest.fin} \textbf{: The rest of $\left(A_i\otimes A_j\oplus A_j\otimes A_i\right)_+$; $A$ finite-dimensional.} The present portion of the proof proposes to show that
    \begin{equation}\label{eq:symsln}
      \spn\left\{\left(a\otimes 1_j\right)^+\ :\ a\in \fg_{i},\ 1_j\text{ unit of }A_j,\ j\ne i\right\}
    \end{equation}
    is contained in \Cref{eq:spn.mp}. For $S\subseteq I$, consider projections $p=p_S\in P(A_S)\subseteq P(A)$. We have
    \begin{equation*}
      p\otimes (1-p)
      =
      p_S\otimes (1_S-p_S)+p_S\otimes 1_{\widehat{S}}
      \in
      \left(A_S\otimes A_S\right)
      \oplus
      \left(A_S\otimes A_{\widehat{S}}\right).
    \end{equation*}
    Similarly, if $1-p=1_S-p_S \in P(A_S)\subseteq P(A)$ instead, the element
    \begin{equation*}
      p_S\otimes (1_S-p_S)+1_{\widehat{S}}\otimes (1_S-p_S)
      \in
      \left(A_S\otimes A_S\right)
      \oplus
      \left(A_{\widehat{S}}\otimes A_{S}\right)
    \end{equation*}
    belongs to the target space. Their difference
    \begin{equation}\label{eq:ps1ps}
      \delta_{p_S}
      :=
      p_S\otimes 1_{\widehat{S}}
      -
      1_{\widehat{S}}\otimes (1_S-p_S)
      \in
      \left(A_S\otimes A_{\widehat{S}}\right)\oplus \left(A_{\widehat{S}}\otimes A_{S}\right)
    \end{equation}
    does too, and I claim that the span
    \begin{equation*}
      \spn\left\{\delta_{p_S}\text{ for varying $S$ and $p_S$}\right\}
      \le
      A\oplus A
      \cong
      \left(A\otimes \bC\right)\oplus \left(\bC\otimes A\right)
      \le
      A\otimes A
    \end{equation*}
    contains \Cref{eq:symsln}. This amounts to showing that for $0\ne a\in \fg_{i}\le A_i$ fixed throughout, no non-zero element 
    \begin{equation}\label{eq:syma}
      \sum_{j\ne i} c_j \left(a\otimes 1_j\right)^+
      ,\quad
      c_j\in \bC\text{ not all }0
    \end{equation}
    is $\braket{-\mid -}$-orthogonal to all $\delta_{p_S}$. To that end, we will set $S=\widehat{j}$ in \Cref{eq:ps1ps} for fixed $j\ne i$. Observe that for any $p\in P(A_{\widehat{j}})$ with $\tau(p_i a)\ne 0$ we have
    \begin{equation}\label{eq:delp.non.orth}
      \Braket{\delta_{p} \mid \left(a\otimes 1_j\right)^+}
      \xlongequal[\tau(a)=0]{\text{\Cref{eq:ftrc}, \Cref{eq:ps1ps}}}
      2\alpha_j \tau(p_i a)
      \ne
      0
      \quad\text{and}\quad
      \begin{aligned}            
        \braket{\delta_{p} \mid A_i\otimes \bC_{j'}}
        &=0\\
        \braket{\delta_{p} \mid \bC_{j'} \otimes A_i}
        &=0\\ 
      \end{aligned}
      ,\quad\forall j'\ne i,j
    \end{equation}
    This suffices to conclude: a non-zero element \Cref{eq:syma} will have a non-zero component in at least one $A_i\otimes \bC_j\oplus \bC_j\otimes A_i$, and that component will ensure non-orthogonality to some $\delta_p$ by \Cref{eq:delp.non.orth}.

    Dwelling on the finite-dimensional case for a while longer, the present argument in fact proves slightly more than has been claimed.
    
  \item\label{item:pr:p1p.spn.sym:rest.fin.bis} \textbf{: For finite-dimensional $A$ we have
      \begin{equation}\label{eq:a1j.dps}
        \left(\forall i\ne j\right)
        \left(\forall a\in \fg_i\le A_i\right)
        \quad:\quad
        \left(a\otimes 1_j\right)^+
        \in
        \spn\left\{\delta_{p_S}\ :\ p_{j'} = 1_{j'}\text{ or }0,\ \forall j'\ne i\right\}.
      \end{equation}
    } Step \Cref{item:pr:p1p.spn.sym:rest.fin} establishes some relation
    \begin{equation}\label{eq:a1j.lcomb}
      \left(a\otimes 1_j\right)^+
      =
      \sum_{\substack{S\subseteq I\\p_S\in P(A_S)}}c_S\delta_{p_S}
    \end{equation}
    which we can simply average against the Haar measure of $\U_{\widehat{i}}:=\prod_{s\ne i}\U_s$. In matrix algebras we have
    \begin{equation*}
      \int_{\U(n)}gpg^{-1}\ \mathrm{d}\mu(g)
      \in
      \bC\le M_n
      ,\quad
      \forall p\in P(M_n),
    \end{equation*}
    so said averaging will leave the left-hand side of \Cref{eq:a1j.lcomb} invariant while transforming the summands of the right-hand side into scalar multiples of the types of $\delta_{P_S}$ \Cref{eq:a1j.dps} refers to. 
    
  \item \textbf{: \Cref{eq:symsln} $\subseteq$ \Cref{eq:spn.mp} in general.} Deducing the present claim from \Cref{item:pr:p1p.spn.sym:rest.fin.bis} is very similar in spirit to the passage from \Cref{item:le:cnst:pf.fd} to \Cref{item:le:cnst:pf.gen} in the proof of \Cref{le:cnst}. For a \emph{finite}-dimensional
    \begin{equation*}
      A'=A'_{\alpha}\times A'_{\beta}\times \cdots
      ,\quad
      A'_{\bullet}\text{ matrix algebras}
      ,\quad
      A'_{\alpha}\cong A_i
    \end{equation*}
    and write, per step \Cref{item:pr:p1p.spn.sym:rest.fin.bis},
    \begin{equation*}
      \left(a\otimes 1_{\beta}\right)^+
      =
      \sum_{F\in 2^{\left\{\alpha,\beta\cdots\right\}}}
      c_F\delta_{p_F}
      \quad\text{with}\quad
      \begin{aligned}
        p_F&\in P(A'_F)\\
        p_{\gamma}&\in \{0,1\}\subset A'_{\gamma}\text{ for }\gamma\ne \alpha. 
      \end{aligned}
    \end{equation*}
    Now consider a Boolean-algebra embedding
    \begin{equation*}
      2^{\alpha,\beta\cdots}
      \lhook\joinrel\xrightarrow{\quad\psi\quad}
      2^I
      ,\quad
      \psi\alpha=i\quad\text{and}\quad \psi\beta=j,
    \end{equation*}
    whence
    \begin{equation*}
      \left(a\otimes 1_{j}\right)^+
      =
      \left(a\otimes 1_{\psi\beta}\right)^+
      =
      \sum_{F\in 2^{\left\{\alpha,\beta\cdots\right\}}}
      c_F\delta_{p_{\psi F}}
      \quad
      \left(\in\text{ right-hand side of \Cref{eq:a1j.dps}}\right),
    \end{equation*}
    where $p_{\psi F}\in P(A)$ is the projection whose $k^{th}$ component $\left(p_{\psi F}\right)_k\in A_k$ is
    \begin{itemize}[wide]
    \item $p_i\in A_i\cong A'_{\alpha}$ if $\alpha\in F$ and $k=i=\psi \alpha$; 

    \item 0 if $k\not\in \psi F$ or $k\in \psi \gamma$ for $\gamma\in F$ with $p_{\gamma}=0$;

    \item and 1 if $k\in \psi \gamma$ for $\gamma\in F$ with $p_{\gamma}=1$.  \qedhere
    \end{itemize}
  \end{enumerate}
\end{item:th:p1p.spn:sym}

% % OLD: recast as part of the proof of \Cref{th:p1p.spn}
% % 
% % \begin{proposition}\label{pr:p1p.spn.sym}
% %   In the context of \Cref{th:p1p.spn} we have
% %   \begin{equation*}
% %     \begin{aligned}
% %       \left(
% %       \ker\mu\bigcap\ker\mu^{\circ}
% %       \right)_+
% %       &\le
% %         \spn\left\{m_p+\sigma m_p\ |\ p\in P(A)\right\}\\
% %       &=
% %         \spn\left\{m_p\ |\ p\in P(A)\right\}_+
% %         \le
% %         \spn\left\{m_p\ |\ p\in P(A)\right\}.
% %     \end{aligned}    
% %   \end{equation*}
% % \end{proposition}

The antisymmetric counterpart to \Cref{th:p1p.spn} certainly does \emph{not} hold, by contrast to the single-matrix-factor picture of \Cref{pr:p1p.spn.mat}. In fact, ``most'' of $\left(\ker\mu\bigcap\ker\mu^{\circ}\right)_-$ is absent from $\cat{INQ}(A)$.

\begin{proposition}\label{pr:bulk-}
  Let $A$ be either the full or $W^*$-algebra $\ell^{\bullet}(\cX)$ attached to a quantum set $\cX=\left\{X_i\right\}$ and set
  \begin{equation*}
    \fg_i:=\text{traceless elements of }A_i:=\cL(X_i)
    \cong
    \fsl_{\dim X_i}.
  \end{equation*}
  We then have 
  \begin{equation*}
    \cat{INQ}(A)
    \bigcap
    \prod^{\bullet}_{i,j}\left(\fg_i\otimes \fg_j\right)_-
    =
    \left(
      \prod^{\bullet}_{i}\left(\fg_i\otimes \fg_i\right)_-
    \right)
    \bigcap
    \left(\ker\mu\bigcap\ker\mu^{\circ}\right)_-.
  \end{equation*}
\end{proposition}
\begin{proof}
  Because $\fg_i\otimes \fg_j$ has multiplicity 2 in $A\totimes A$ (the symmetric and antisymmetric summands respectively) for $\widetilde{\bullet}\in \{\widehat{\bullet},\ \overline{\bullet}\}$, \Cref{th:p1p.spn} and \Cref{pr:p1p.spn.mat} reduce the claim to simply showing that
  \begin{equation*}
    \forall i\ne j
    \quad:\quad
    \left(\fg_i\otimes \fg_i\right)_-
    \not\le
    \cat{INQ}(A).
  \end{equation*}
  This, in turn, follows from the fact that $\left(\fg_i\otimes \fg_i\right)_-$ is $\braket{-\mid -}_{\tau}$-orthogonal to $m_p$ for any $p\in P(A)$ and any (possibly non-faithful) tracial state $\tau$ on $A$.
\end{proof}

\section{Asides on jointly-continuous conjugation actions}\label{se:jcont.wast}

\begin{definition}\label{def:qset.virt.cls}
  A quantum set $\cX=\left\{X_i\right\}$ is \emph{virtually classical} if at most finitely many $\cL(X_i)$ are non-abelian. 
\end{definition}

\begin{proposition}\label{pr:cont.act.virt.cls}
  For a quantum set $\cX=\left\{X_i\right\}$ the following conditions are equivalent.
  \begin{enumerate}[(a),wide]
  \item\label{item:pr:cont.act.virt.cls:virt.cls} $\cX$ is virtually classical.

  \item\label{item:pr:cont.act.virt.cls:conj.all.n} The conjugation actions on $A^{\ootimes n}$, $A:=\ell^{\infty}(\cX)$ of $\U:=\U(A)$ are jointly continuous for the product topology on $\U$ and the $\sigma$-weak topology on $A^{\ootimes n}$.
    
  \item\label{item:pr:cont.act.virt.cls:conj1} As in \Cref{item:pr:cont.act.virt.cls:conj.all.n}, for the single conjugation action on $A$. 

  \item\label{item:pr:cont.act.virt.cls:conj.some.n} As in \Cref{item:pr:cont.act.virt.cls:conj.all.n}, for \emph{some} (rather than all) $n\in \bZ_{\ge 1}$.
  \end{enumerate}
\end{proposition}
\begin{proof}
  The implications \Cref{item:pr:cont.act.virt.cls:conj.all.n} $\Rightarrow$ \Cref{item:pr:cont.act.virt.cls:conj1} $\Rightarrow$ \Cref{item:pr:cont.act.virt.cls:conj.some.n} are obvious; we address two others.  

  \begin{enumerate}[label={},wide]

  \item\textbf{\Cref{item:pr:cont.act.virt.cls:virt.cls} $\Rightarrow$ \Cref{item:pr:cont.act.virt.cls:conj.all.n}:} Decompose $A$ as $A\cong B\times A_{ab}$, with the latter factor collecting the 1-dimensional factors $\cL(X_i)$, $\dim X_i=1$. $\U$ acts trivially on $A_{ab}$, so the action on $A^{\ootimes n}$ identifies with one on an $\ell^{\infty}$-sum
    \begin{equation}\label{eq:linf.botimes}
      \prod^{\cat{W}^*}_{j\in J} B^{\ootimes m_j}
      ,\quad
      0\le m_j\le n.
    \end{equation}
    If $B$ is finite-dimensional the action (factoring through $\U\xrightarrowdbl{} \U(B)$) is plainly jointly continuous by direct examination.
    
  \item\textbf{\Cref{item:pr:cont.act.virt.cls:conj.some.n} $\Rightarrow$ \Cref{item:pr:cont.act.virt.cls:virt.cls}:} I first claim that for non-trivial normed spaces $\left(E_i,\ \|\cdot\|_i\right)_{i\in I}$, the factor-wise scaling action
    \begin{equation}\label{eq:ti.on.dual}
      \bT^I := \left(\bS^1\right)^I
      \circlearrowright
      \ell^{\infty}\left(E_i',\ \text{dual norm }\|\cdot\|'_i\right)_{i\in I}
    \end{equation}
    (with primes denoting Banach-space duals) is jointly continuous precisely when $I$ is finite.

    One direction is obvious (if $I$ is finite there are no continuity issues), so assume $I$ infinite and fix elements
    \begin{equation*}
      0\ne e_i\in E_i
      ,\quad
      \sum_{i\in I}\|e_i\|_i<\infty.
    \end{equation*}
    For every finite set $F$ of elements in the predual $\ell^1\left(E_i,\ \|\cdot\|_i\right)_{i\in I}$ of $\ell^{\infty}\left(E'_i\right)_{i\in I}$ consider a functional
    \begin{equation*}
      \psi\in
      \ell^{\infty}\left(E'_i\right)_{i\in I}\cong \ell^1\left(E_i\right)'_{i\in I}
      \quad:\quad
      \psi|_F\equiv 0
      \quad\text{and}\quad
      \exists \left(\varepsilon_i\right)_i\in \left\{\pm 1\right\}^I
      \ :\ 
      \psi\left(
        \left(\varepsilon_i e_i\right)_i
      \right)=1.
    \end{equation*}
    This is possible by \emph{Hahn-Banach} \cite[Theorem III.6.2]{conw_fa}, given that the tuples
    \begin{equation}\label{eq:epse}
      \left(\varepsilon_i e_i\right)_i
      \in
      \ell^1\left(E_i\right)_{i\in I}
      ,\quad
      \mathbf{\varepsilon}:=\left(\varepsilon_i\right)_i\in \left\{\pm 1\right\}^I
    \end{equation}
    are linearly independent. The \emph{net} \cite[Definition 11.2]{wil_top} $\left(\psi_F\right)_F$ (with finite sets $F$ ordered by inclusion) converges to 0 $\sigma$-weakly but not in the topology of uniform convergence on orbits of the $\bT^I$-action, for all \Cref{eq:epse} lie on one orbit. The action \Cref{eq:ti.on.dual}, then, cannot be jointly continuous. 

    Claim in hand, assume $\U \circlearrowright A^{\ootimes n}$ jointly continuous. For all
    \begin{equation*}
      A_i\cong M_{d_i}\cong \cL(X_i)
      ,\quad
      d_i:=\dim X_i\ge 2
    \end{equation*}
    there is some circle $\bS^1\cong \bS_i\le \U_i:=\U(A_i)$ with $z\in \bS^1\cong \bS_i$ acting as $z$-scaling on some 1-dimensional space $F_i\le A_i$. The preceding abstract remark applies to the action of the torus
    \begin{equation*}
      \bT^{\left\{i\ :\ d_i\ge 2\right\}}
      :=
      \prod_{d_i\ge 2}\bS_i
      \le
      \U_{\left\{i\ :\ d_i\ge 2\right\}}
    \end{equation*}
    on
    \begin{equation*}
      \ell^{\infty}\left(F_i\right)_{d_i\ge 2}
      \cong
      \ell^{\infty}\left(F_i\otimes \bC\right)_{d_i\ge 2}
      \le
      \prod^{\cat{W}^*}_{d_i\ge 2}A_i\ootimes A_k^{\ootimes (n-1)}
      \le
      A^{\ootimes n}
      \quad
      \left(\text{fixed $k$}\right).
    \end{equation*}
  \end{enumerate}
  This completes the implication circle.
\end{proof}

\begin{remarks}\label{res:joint.cont}
  \begin{enumerate}[(1),wide]
  \item There is an alternative take on the implication \Cref{item:pr:cont.act.virt.cls:virt.cls} $\Rightarrow$ \Cref{item:pr:cont.act.virt.cls:conj.all.n} in the proof of \Cref{pr:cont.act.virt.cls} above. For finite-dimensional $B$ the $\U$-action on \Cref{eq:linf.botimes} is dual to one on the predual
    \begin{equation*}
      \ell^1\left(B^{\otimes m_j}\right)_{j\in J}
      \text{ of \Cref{eq:linf.botimes}}.
    \end{equation*}
    The latter predual action has finitely many isotypic components (with continuous isotypic projections), so the claimed joint continuity follows from \Cref{le:fin.isot.jcont} below.
    
  \item It is certainly possible, in general, for a compact group to act (separately continuously and) isotypically on a locally convex quasi-complete space, but not \emph{jointly} continuously.

    We will adapt the gadget employed in \cite[Problem 52]{hal_hspb_2e_1982} (and on \cite[p.110]{halm_sm_2e_1957}) to produce closed subspaces with non-closed algebraic sum. Consider the embedding $V\xrightarrow{\iota}\widetilde{V}$ of a quasi-complete, incomplete locally convex space into its \emph{completion} \cite[\S 15.3(1)]{k_tvs-1}, and set
    \begin{equation*}
      W_0:=V\oplus \{0\}\le V\oplus\widetilde{V}=:F
      \quad\text{and}\quad
      W_1
      :=
      \text{graph of $\iota$}      
      :=
      \left\{(v,\iota v)\ :\ v\in V\right\}
      \le
      F
    \end{equation*}
    The sum
    \begin{equation*}
      W:=
      W_0+W_1
      =
      V\oplus V
      \le
      F
    \end{equation*}
    is direct ($\iota$ being injective) and henceforth subspace-topologized. Equip $W$ with
    \begin{itemize}[wide]
    \item the $\bS^1$-action $z\mapsto z$ on $W_0$;

    \item the $\bS^1$-action $z\mapsto z^*=\frac 1z$ on $W_1$;

    \item and a $\bZ/2$-action interchanging $W_{0,1}$ via
      \begin{equation*}
        W_0\ni
        (v,0)
        \xleftrightarrow{\quad}
        (v,\iota v)
        \in W_1.
      \end{equation*}
    \end{itemize}
    This results in an isotypic action of the orthogonal group
    \begin{equation*}
      \OR(2)\cong \bS^1\rtimes \left(\bZ/2\right)
      ,\quad
      \bZ/2\text{ acting by inversion},
    \end{equation*}
    separately continuous but certainly not jointly so. Indeed, there are nets
    \begin{equation*}
      W_0\oplus W_1
      \ni
      w_{\lambda}=w_{0\mid\lambda}+w_{1\mid\lambda}
      \xrightarrow[\quad\lambda\quad]{\quad}
      0
      ,\quad
      w_{0,1\mid\lambda}\not \in\text{some neighborhood }U\ni 0
    \end{equation*}
    by construction. One may as well assume $p_{\lambda}:=p\left(w_{0\mid \lambda}+w_{1\mid \lambda}\right)\xrightarrow[\lambda]{}\infty$ for some continuous \emph{seminorm} \cite[\S 18.1]{k_tvs-1} $p$ on $F$, and 
    \begin{equation*}
      e^{-2\pi i p_{\lambda}}
      \triangleright
      \left(w_{0\mid\lambda}+w_{1\mid\lambda}\right)
      =
      e^{-2\pi i p_{\lambda}} w_{0\mid\lambda}
      +
      e^{2\pi i p_{\lambda}} w_{1\mid\lambda}
    \end{equation*}
    will not converge to $0\in W_0\oplus W_1$.
  \end{enumerate}
\end{remarks}

\begin{lemma}\label{le:fin.isot.jcont}
  If a separately continuous action $\triangleright : \G\circlearrowright E$ of a compact group on a locally convex space has finitely many isotypic components, then the dual action $\G\circlearrowright E'$ on the weak dual is jointly continuous. 
\end{lemma}
\begin{proof}
  We have separate continuity in any case \cite[p.VIII.11, Proposition 3(i)]{bourb_int_en_7-9}, and the finite-isotypic-component assumption ensures that $\G$-orbits in $E$ are finite-dimensional, so the topology of pointwise converge on $E'$ coincides with that of uniform convergence on those $\G$-orbits. 
\end{proof}

\Cref{th:cofin.ablnz} will supersede \Cref{pr:cont.act.virt.cls}, relying on its proof. The statement refers to \emph{dual topologies} on a von Neumann algebra $A$, by which we mean any locally convex topology $\cT$ with the property that the dual $(A,\cT)'$ is the \emph{predual} \cite[Corollary III.3.9]{tak1} $A_*$. By the celebrated \emph{Arens-Mackey theorem} \cite[\S 21.4(2)]{k_tvs-1}, these are precisely the locally convex topologies
\begin{itemize}[wide]
\item at least as fine as the $\sigma$-weak;

\item and at most as fine as the \emph{Mackey topology} \cite[\S 21.4]{k_tvs-1}.
\end{itemize}
They all agree on the unitary group $\U(A)$:
\begin{itemize}[wide]
\item Since we are restricting attention to the bounded set $\U(A)$ there is no distinction \cite[Lemma II.2.5]{tak1} between the weak, strong and strong$^*$ topologies of \cite[\S II.2]{tak1} and their respective $\sigma$ counterparts.

\item The weak and strong topologies coincide by \cite[Solution 20]{hal_hspb_2e_1982} and hence so do the weak and strong$^*$ because the former is $*$-invariant \cite[Solution 110]{hal_hspb_2e_1982}.

\item And finally, strong$^*$ = Mackey on bounded sets \cite[Theorem II.7]{zbMATH03253462}. 
\end{itemize}

In the following statement the symbol $\totimes$ stands for either the maximal or the spatial von Neumann tensor product \cite[\S\S *.1 and 8.2]{zbMATH03248454}. 

\begin{theorem}\label{th:cofin.ablnz}
  For a von Neumann algebra $A$ the following conditions are equivalent. 

  \begin{enumerate}[(a),wide]
  \item\label{item:th:cofin.ablnz:virt.cls} The kernel of the abelianization $A\xrightarrowdbl{} A_{ab}$ is finite-dimensional. 

  \item\label{item:th:cofin.ablnz:conj.all.n} The conjugation actions on $A^{\totimes n}$ of $\U:=\U(A)$ are jointly continuous for the product topology on $\U$ and the $\sigma$-weak topology on $A^{\ootimes n}$.

  \item\label{item:th:cofin.ablnz:conj1} As in \Cref{item:th:cofin.ablnz:conj.all.n}, for the single conjugation action on $A$. 

  \item\label{item:th:cofin.ablnz:conj.some.n} As in \Cref{item:th:cofin.ablnz:conj.all.n}, for \emph{some} (rather than all) $n\in \bZ_{\ge 1}$.
  \end{enumerate}
\end{theorem}
\begin{proof}
  Condition \Cref{item:th:cofin.ablnz:virt.cls} simply means that $A\cong A_{ab}\times A'$ for some finite-dimensional $A'$. In terms of the \emph{type classification/decomposition} of \cite[Theorems V.1.19 and V.1.27]{tak1}, it also means that
  \begin{itemize}[wide]
  \item $A$ is of type I;
  \item and the non-abelian factor in the resulting \cite[Theorem V.1.27]{tak1} decomposition
    \begin{equation}\label{eq:gen.t1}
      A\cong \prod_{\text{cardinals $\lambda$}}^{\cat{W}^*}A_{\lambda}\ootimes \cL(H_{\lambda})
      ,\quad
      \begin{aligned}
        A_{\lambda}&\quad\text{abelian}\\
        \dim H_{\lambda} &= \lambda
      \end{aligned}
    \end{equation}
    the factors with $\lambda\ge 2$ are finite in number and all finite-dimensional. 
  \end{itemize}
  The implications \Cref{item:th:cofin.ablnz:conj.all.n} $\Rightarrow$ \Cref{item:th:cofin.ablnz:conj1} $\Rightarrow$ \Cref{item:th:cofin.ablnz:conj.some.n} are again formal, and \Cref{item:th:cofin.ablnz:virt.cls} $\Rightarrow$ \Cref{item:th:cofin.ablnz:conj.all.n} follows essentially verbatim as in the proof of \Cref{pr:cont.act.virt.cls}. As to \Cref{item:th:cofin.ablnz:conj.some.n} $\Rightarrow$ \Cref{item:th:cofin.ablnz:virt.cls}, here too we can repurpose the argument employed in proving \Cref{pr:cont.act.virt.cls}.
  
  Indeed, joint continuity will fail whenever there is a von Neumann subalgebra
  \begin{equation*}
    \prod^{\cat{W}^*}_{i\in I}M_{n_i}
    \le
    A
    ,\quad
    \left\{i\in I\ :\ n_i\ge 2\right\}
    \text{ infinite},
  \end{equation*}
  and such embeddings do exist in types $II$ and $III$ as well as type $I$ when \Cref{eq:gen.t1} either has infinitely many factors or at least one with infinite $\lambda$. 
\end{proof}

In this context of unitary symmetries, note also the following characterization of hereditary atomicity supplementing \cite[Proposition 5.4]{zbMATH07287276}.

\begin{lemma}\label{le:her.at.u.cpct}
  A von Neumann algebra is hereditarily atomic if and only if its unitary group is compact in the dual topology. 
\end{lemma}
\begin{proof}
  The implication $(\Rightarrow)$ is immediate, and the converse $(\Leftarrow)$ follows from (the proof of) \cite[Proposition 5.4]{zbMATH07287276}: the unitary group of $L^{\infty}:=L^{\infty}([0,1])$ is of course not compact, so the $W^*$-algebra in question can contain no ($\sigma$-weakly closed) copy of $L^{\infty}$. 
\end{proof}

%%%%%%%%%%%%%%%%%%%%%%%%%%%%%%%% 
%%%%%%%%%%%%%%%%%%%%%%%%%%%%%%%%

%\newpage

\addcontentsline{toc}{section}{References}
%\bibliography{bib}{}
%\bibliographystyle{plain}

\def\polhk#1{\setbox0=\hbox{#1}{\ooalign{\hidewidth
  \lower1.5ex\hbox{`}\hidewidth\crcr\unhbox0}}}

\Addresses

\end{document}